\documentclass{article} 
\usepackage{latexsym}
\usepackage{amssymb}
\usepackage{amsmath,amsthm}
\usepackage{mathrsfs}
\usepackage{hyperref}
\usepackage{graphicx}
\usepackage{verbatim}
\usepackage{pb-diagram,amsfonts,amsmath}

\usepackage[all]{xy}

\newtheorem{theorem}{Theorem}[section]
\newtheorem{exa
mple}{example}[section]
\newtheorem{lemma}[theorem]{Lemma}

\newtheorem{corollary}[theorem]{Corollary}
\newtheorem{proposition}[theorem]{Proposition}

\theoremstyle{definition}
\newtheorem{definition}{Definition}[section]

\theoremstyle{remark}
\newtheorem{remark}[theorem]{Remark}

\newcommand\N{\mathbb{N}}

\newcommand\Z{\mathbb{Z}}
\newcommand\R{\mathbb{R}}

\newcommand\T{\mathcal{T}}
\newcommand\M{\mathcal{M}}
\renewcommand\S{\mathcal{S}}

\newcommand\AT{\overline{\mathcal{T}}}
\newcommand\AM{\overline{\mathcal{M}}}
\renewcommand\P{\mathbb{P}}
\newcommand\PP{{{P}_f}}
\newcommand\Sphere{{\mathbb{S}^2}}

\renewcommand{\mod}{\mbox{\rm mod }}

\newcommand{\sigmat}{\tilde{\sigma}}
\newcommand{\fh}{\hat{f}}
\newcommand{\sm}{\setminus}
\newcommand{\eps}{\varepsilon}

{\par\noindent\textit{Proof of (#1)}%
} 

\def\Teich{Teich\-m\"uller }

\begin{document}
\large


\title{Topological characterization of canonical Thurston obstructions}
\author{Nikita Selinger}

\date{\today}
\maketitle

\begin{abstract}
Let $f$ be an obstructed Thurston map with canonical obstruction $\Gamma_f$. We prove the following generalization of Pilgrim's conjecture:  if the first-return map $F$ of a periodic component $C$ of the topological surface obtained from the sphere by pinching the curves of  $\Gamma_f$ is a Thurston map then the canonical obstruction of  $F$ is empty. Using this result, we give a complete topological characterization of canonical Thurston obstructions. 

\end{abstract}

\tableofcontents


\section{Introduction}

Thurston's characterization theorem for branched self-covers of the topological 2-shpere $\Sphere$ \cite{DH} gives a pure topological criterion whether a branched cover $f$ can be geometrized, i.e. if the topological sphere admits an invariant (in an appropriate sense) with respect to $f$  conformal structure. The theorem states that the latter is not possible only if there exists a Thurston obstruction for $f$ which is a collection of simple closed curves on $\Sphere$ subject to certain conditions.

Pilgrim refined the proof of Thurston's theorem \cite{P} by introducing a notion of canonical Thurston obstructions and showing that a branched cover $f$ with hyperbolic orbifold is obstructed if and only if the canonical obstruction  $\Gamma_f$ of $f$ is not empty. Cutting the initial branched cover $f$ into pieces along the curves of the canonical obstruction, Pilgrim deduces Canonical Decomposition Theorem \cite{P1} for branched covers. Canonical geometrization of an obstructed branched cover can now be constructed by geometrizing each piece in the canonical decomposition.

It is straightforward that any curve $\gamma \in \Gamma_f$ does not intersect any other curve of any other obstruction. The converse, however, turns out to be false (but almost true). It is also immediate to see that if a minimal obstruction has the leading eigenvalue strictly greater than one, then it is a subset of the canonical obstruction. These, however, were the only two known topological properties of canonical obstructions.
 In this article, we study further properties of canonical obstructions, using the techniques developed in \cite{S11}, and give their complete topological description  (see Theorem~\ref{thm:main}).

In the recent work of Bonnot, Braverman, Yampolsky \cite{BBY}, the authors show that there exists an algorithm that can find geometrization of an unobstructed branched cover (or prove that the cover is obstructed). However, the question of algorithmically geometrizing arbitrary branched covers remains open. The first step was made in \cite{S11} where the author proved Pilgrim's conjecture from \cite{P1}. We generalize the statement of Pilgrim's conjecture for all Thurston maps (also with parabolic orbifolds) and give a new proof to this conjecture (see Theorem~\ref{thm:KevinConjecture2}). The theorem states that in a decomposition of a Thurston map along its canonical obstruction, the canonical obstruction of every periodic component is empty.

We work in the setting of Thurston maps that are defined as pairs of a postcritically finite branched cover $f$ and a forward invariant marked set $Q$, which contains the postcritical set $\PP$ of $f$. Maps of these type naturally arise in applications (see, for example, \cite{cheritat}). The importance of working in this generality comes from the fact that the class of Thurston maps together with homeomorphisms of spheres with marked points is closed under decompositions while if you decompose a classical postcritically finite branched cover you may easily get extra marked points. Classification of Thurston maps with hyperbolic orbifolds can be obtained by almost the same argument as the original Thurston's theorem. One can write a full list of postcritically finite maps with parabolic orbifolds (see \cite{DH}) up to Thurston equivalence with respect to the postcritical set. The classification of general Thurston maps with parabolic orbifolds is yet to be developed, however. The present article provides the first step in this direction.

As an example consider a flexible Latt\`es map $f$ which is the  quotient of a map $F(z)=4z$ defined on the plane by the group generated by $z\mapsto z+1, z\mapsto z+i, z\mapsto-z$. The point $\tilde{q}=(1+i)/3$ projects onto a fixed point $q$ of $f$. The pair $(f, Q=\PP\cup\{q\})$ is a rational Thurston map with parabolic orbifold. If we postcompose $f$ with an element of the mapping class group of $(\Sphere, Q)$, say a Dehn twist $T$ around a curve surrounding $q$ and one of the postcritical points, we get a new Thurston map $(T\circ f, Q)$ with parabolic orbifold. One may ask a lot of questions about this new map, such as whether or not it is equivalent to a rational map. In the particular case described above, if $(T\circ f,Q)$ is equivalent to a rational map, then $(T\circ f,Q)$ is equivalent to $(f,Q')$ with $Q'=\PP\cup \{q'\}$ where $q'$ is  a possibly different fixed point of $f$. If $(T\circ f,Q)$ is not equivalent to a rational map, then one expects that the canonical obstruction of must be non-empty. A thorough investigation of questions of this kind will be based on the results obtained in this article.

 We believe that the results of this article will also help in proving that both the  canonical decomposition of a branched cover and the geometrization thereof can be found algorithmically. This will lead, in particular, to an algorithm that can check equivalence of two Thurston maps.

{\bf Acknowledgments.}
This article is based on the results from the author's PhD thesis \cite{thesis}. I thank Dima Dudko, Adam Epstein, John Hubbard, Sarah Koch, Misha Lyubich, Daniel Meyer, Kevin Pilgrim, and Dierk Schleicher for useful conversations. I thank the referee for useful remarks and suggesting a shorter proof for Theorem~\ref{thm:2222obstruction}.
This research was partially supported by the Deutsche Forschungsgemeinschaft.

\section{Basic definitions}
\label{sec:basic}

The main setup is the same as in \cite{DH,S11}.

Let $f$ be  an orientation-preserving branched self-cover of degree $d_f \ge2$ of the 2-dimensional sphere $\Sphere$. The \emph{critial set} $\Omega_f$ is the set of all points $z$ in $\Sphere$ where the local degree of $f$ is greater than 1. The \emph{postcritial set} $\PP$ is the union of all forward orbits of $\Omega_f$, i.e. $\PP= \cup_{i\ge1} f^i(\Omega_f).$  A branched cover $f$ is called \emph{postcritially finite} if $\PP$ is finite. More generally, a pair $(f,Q_f)$ of a branched cover $f \colon \Sphere \to \Sphere$ and a finite set $Q_f \subset \Sphere$ is called a \emph{Thurston map} if $Q_f$ is forward invariant and contains all critical values of $f$ (and, hence, contains $\PP$). 


Two Thurston maps $f$ and $g$ are \emph{Thurston equivalent} if and only if there exist two homeomorphisms $h_1,h_2\colon \Sphere \to \Sphere$ such that the diagram
\[
\begin{diagram}
\node{(\Sphere,Q_f)} \arrow{e,t}{h_1} \arrow{s,l}{f} \node{(\Sphere,Q_g)}
\arrow{s,l}{g}
\\
\node{(\Sphere,Q_f)} \arrow{e,t}{h_2} \node{(\Sphere,Q_g)}
\end{diagram}
\]
commutes, $h_1|_{Q_f}=h_2|_{Q_f}$, and
$h_1$ and $h_2$ are homotopic relative to $Q_f$.

A simple closed curve $\gamma$ is called \emph{essential} if every component of $\Sphere \sm \gamma$ contains at least two points of $Q_f$. We consider essential simple closed curves up to free homotopy in $\Sphere\sm Q_f$. A \emph{multicurve} is a finite set of pairwise disjoint and non-homotopic essential simple closed curves. Denote by $f^{-1}(\Gamma)$ the multicurve consisting of all essential preimages of curves in $\Gamma$. A multicurve $\Gamma=(\gamma_1,\ldots,\gamma_n)$ is called \emph{invariant} if each
component of $f^{-1}(\gamma_i)$ is either non-essential, or it is
homotopic (in $\Sphere\sm Q_f$) to a curve in $\Gamma$ (i.e. $f^{-1}(\Gamma) \subseteq \Gamma$).
We say that $\Gamma$ is \emph{completely invariant} if $f^{-1}(\Gamma) = \Gamma$. 

Every multicurve $\Gamma$ has its associated \emph{Thurston
matrix} $M_\Gamma=(m_{i,j})$ with
\[
m_{i,j}=\sum_{\gamma_{i,j,k}} (\deg f|_{\gamma_{i,j,k}}\colon
\gamma_{i,j,k}\to \gamma_j)^{-1}
\]
where $\gamma_{i,j,k}$ ranges  through all preimages of $\gamma_j$
that are homotopic to $\gamma_i$. Since all entries of $M_\Gamma$
are non-negative real, the leading eigenvalue $\lambda_\Gamma$ of
$M_\Gamma$ is real and non-negative (see Corollary~\ref{cor:largestEV}). 
\begin{remark}
Note that to define the Thurston matrix $M_\Gamma$, we do not require that $\Gamma$ be invariant.
\end{remark}

A multicurve $\Gamma$ is a
\emph{Thurston obstruction} if $\lambda_\Gamma\ge 1$. A Thurston obstruction $\Gamma$ is \emph{minimal} if no proper subset of $\Gamma$ is itself an obstruction. We call $\Gamma$ a \emph{simple} obstruction (compare \cite{P}) if no permutation of the curves in $\Gamma$ puts $M_\Gamma$ in the block form
$$ M_\Gamma = \left( 
\begin{array}{cc}
	M_{11}  & 0 \\
	M_{21}  & M_{22}
\end{array}
 \right),
$$ 
where  the leading eigenvalue of $M_{11}$ is less than $1$. If such a permutation exists, it follows that $M_{22}$ is a Thurston matrix of a multicurve with the same leading eigenvalue as $M_\Gamma$. It is, thus, evident that every obstruction contains a simple one. 

\begin{proposition}
\label{prop:positive}
  A multicurve $\Gamma$ is a simple obstruction if and only if there exists a  vector $v>0$ such that $M_\Gamma v \ge v$.
\end{proposition}

\begin{proof}
Let $v$ be a non-negative vector such that $M_\Gamma v \ge v$ with a maximal possible number of positive components. Applying a permutation if necessary, we assume that $v=(0,v_1)^T$ where $v_1>0$. If we write $M_\Gamma$ in the corresponding block form we get:
$$ M_\Gamma = \left( 
\begin{array}{cc}
	M_{11}  & 0 \\
	M_{21}  & M_{22}
\end{array}
 \right).
$$   If the leading eigenvalue of $M_{11}$ is greater than 1, then there exists a non-negative eigenvector $v_2$ for $M_{11}$ such that $M_\Gamma (v_2\,  v_1)^T \ge  (v_2\,  v_1)^T$. The choice of  $v=(0,v_1)^T$ implies that either $v$ is positive, or the leading eigenvalue of $M_{11}$ is less than 1, and $\Gamma$ is not a simple obstruction. 

On the other hand, it is clear that if there exists a positive vector $v$ with $Mv \ge v$, then the leading eigenvalue of $M$ is at least 1. Therefore, if there exists a positive vector $v$ with $M_\Gamma v \ge v$, then $M_\Gamma$ cannot be written in a block form as above so that the leading eigenvalue of $M_{11}$ is less than 1.
\end{proof}

Note that every minimal obstruction is simple, and that a union of two disjoint simple obstructions is also simple (but not minimal). A \emph{Levy cycle} is a multicurve $\Gamma=\{\gamma_1,\ldots,\gamma_n\}$ such that for every $i=\overline{1,n}$, there exists a preimage component of $\gamma_i$ that is homotopic to $\gamma_{i+1}$ and is mapped to $\gamma_i$ by $f$ with degree 1 (we set $\gamma_{n+1}=\gamma_1$). Every Levy cycle is a Thurston obstruction.

Thurston's original characterization theorem is formulated as follows:

\begin{theorem}[Thurston's Theorem \cite{DH}]
\label{thm:Thurston}  A postcritically finite  bran\-ched cover
$f\colon\Sphere\to\Sphere$ with hyperbolic orbifold is either
Thurston-equivalent to a rational map $g$ (which is then
necessarily unique up to conjugation by a Moebius transformation), or
$f$ has a Thurston obstruction.
\end{theorem}

\begin{remark} In the original formulation in \cite{DH}, a Thurston obstruction was required to be invariant. Omitting this requirement makes the statement of the theorem weaker in one direction and stronger in the other direction. However, in \cite{thesis} we showed that if there exists a Thurston obstruction for $f$, then there also exists a simple invariant obstruction (see Proposition~\ref{prop:simple}).
\end{remark}

General rigorous definition of orbifolds and their Euler
characteristic can be found in \cite{M}. In our case, there is a unique and straightforward way to construct the minimal
function $v_f$ of all functions  $v:\Sphere \to \N \cup \{\infty\}$
satisfying the following two conditions:
\begin{enumerate}
 \item[(i)] $v(x)=1$ when $x \notin P_f$;
 \item[(ii)] $v(x)$ is divisible by $v(y) \deg_y f$ for all  $y \in f^{-1}(x)$.
\end{enumerate}
We say that $f$ has hyperbolic orbifold $O_f=(\Sphere,v_f)$ if the
Euler characteristic of $O_f$
\begin{equation}
\chi(O_f)=2- \sum_{x \in P_f} \left( 1- \frac{1}{v_f(x)} \right) 
\label{eq:euler}
\end{equation} is less than 0, and  parabolic orbifold otherwise. We discuss Thurston maps with parabolic orbifolds in more detail in Section~\ref{sec:parabolic}.

\section{Thurston's pullback map and canonical obstructions}
\label{sec:Teich}

Let $\T_f$ be the \Teich space modeled on the marked surface $(\Sphere,Q_f)$ and $\M_f$  be the corresponding moduli space. We write $\tau =\left\langle h \right\rangle$ if a point $\tau$ is represented by a diffeomorphism $h$. Correspondingly, points of $\M_f$ are represented by $h(Q_f)$ modulo post-composition with Moebius transformations. Denote by $\pi:\T_f \to \M_f$ the canonical covering map which sends $h \mapsto h|_{Q_f}$. The (pure) mapping class group of $(\Sphere,Q_f)$ is canonically identified with the group of deck transformations of $\pi$. For more background on \Teich spaces see, for example, \cite{IT,H}. 

Consider an essential simple closed curve $\gamma$ in $(\Sphere,Q_f)$. For each complex structure $\tau$ on $(\Sphere,Q_f)$, there exists a unique geodesic $\gamma_\tau$ in the homotopy class of $\gamma$.  We denote by $l(\gamma,\tau)$ the length of the geodesic $\gamma_\tau$ homotopic to $\gamma$ on the Riemann surface corresponding to $\tau \in \T_f$. This defines a continuous function from $\T_f$ to $\R_+$ for any given $\gamma$. Moreover, $\log l(\gamma,\tau)$ is a Lipschitz function with Lipschitz constant $1$ with respect to the \Teich metric (see \cite[Theorem 7.6.4]{H}; note that in \cite[Proposition 7.2]{DH} the constant is 2 because of a different normalization of the \Teich metric).
We will use the same notation $l(\gamma, R)$ for the hyperbolic length of a curve $\gamma$ in a hyperbolic surface $R$. Recall that, by the Collar Lemma, the length of a simple closed geodesic $\gamma$ on a hyperbolic Riemann surface $R$ is closely related to the supremum $M(\gamma,R)$ of moduli of all annuli on $R$ that are homotopic to this geodesic, namely $$\frac{\pi}{l(\gamma,R)}-1 < M(\gamma,R) < \frac{\pi}{l(\gamma,R)}.$$


The Thurston's pullback $\sigma_f$ is defined as follows. Suppose $\tau \in \T_f$ is represented by a homeomorphism $h_\tau$. Consider the following diagram:

\begin{equation}
\begin{diagram}
\node{(\Sphere,Q_f)}  \arrow{s,l}{f} 
\\
\node{(\Sphere,Q_f)} \arrow{e,t}{h_\tau} \node{(\P,h_\tau(Q_f))}
\end{diagram}
\label{dg1}
\end{equation}

We can define a complex structure on $(\Sphere, Q_f)$ by pulling  back the complex structure from $\P$ by $h_\tau \circ f$. By the uniformization theorem,  the constructed Riemann surface  is isomorphic to the Riemann sphere $\P$;  let $h_1$ be a conformal isomorphism between $(\Sphere,Q_f)$, endowed with the pulled back complex structure, and $\P$. Set $\sigma_f(\tau)=\tau_1$ where $\tau_1$ is the point represented by $h_1$.


Now we can complete the previous diagram by setting $f_\tau=h_\tau \circ f \circ h_1^{-1}$ so that it commutes:

\begin{equation}
\begin{diagram}
\node{(\Sphere,Q_f)} \arrow{e,t}{h_1} \arrow{s,l}{f} \node{(\P,h_1(Q_f))}
\arrow{s,l}{f_\tau}
\\
\node{(\Sphere,Q_f)} \arrow{e,t}{h_\tau} \node{(\P,h_\tau(Q_f))}
\end{diagram}
\label{dg2}
\end{equation}

Note that from definition of $f_\tau$, it follows that $f_\tau$ respects the standard complex structure $\mu_0$ and, hence, is rational. When we choose a representing homeomorphism $h_\tau$, we have the freedom to post-compose $h_\tau$ with any Moebius transformation; similairly, the choice of $h_\tau$ defines $h_1$ up to a post-composition by Moebius transformation. Thus, $f_\tau$ is defined up to pre- and post-composition by Moebius transformations.


The following proposition \cite[Proposition 2.3]{DH} relates dynamical properties of $\sigma_f$ to the original question.

\begin{proposition}
\label{prop:fixedpts}
A Thurston map $f$ is equivalent to a rational function if and only if $\sigma_f$ has a fixed point.
\end{proposition}


The \emph{canonical} obstruction $\Gamma_f$ is the set of all homotopy classes of curves $\gamma$ that satisfy $l(\gamma,\sigma_f^n(\tau)) \to 0$ for all (or, equivalently, for some) $\tau \in \T_f$. In \cite{thesis} we proved the following (see also \cite{P}).

\begin{proposition}
\label{prop:simple}
If $\Gamma_f$ is not empty then it is a simple completely invariant Thurston obstruction.
\end{proposition}

The following is immediate.

\begin{proposition}
\label{prop:iterate}
$\Gamma_f=\Gamma_{f^r}$.
\end{proposition}

\begin{proof} Take any $\tau \in \T_f$. If $\gamma \in \Gamma_f$ then $l(\gamma, \sigma^n_f(\tau) ) \to 0$. In particular, $l(\gamma,\sigma^{nr}_f(\tau) ) = l(\gamma,\sigma^{n}_{f^r}(\tau) )\to 0$, hence $\gamma \in \Gamma_{f^n}$. 

Set $D=d_T(\tau,\sigma_f(\tau))$. Since $\sigma_f$ is weakly contracting and $\log l(\cdot, \gamma)$ is 1-Lipschitz, we have $l (\gamma,\sigma^{nr+s}_f(\tau)) \le e^{Ds} l( \gamma,\sigma^{nr}_f(\tau))$. Therefore, if $l( \gamma,\sigma^{nr}_f(\tau)) \to 0$ then  $l( \gamma,\sigma^n_f(\tau)) \to 0$ as well.
 \end{proof}

The next theorems are due to Kevin Pilgrim \cite{P}. 
\begin{theorem}[Canonical Obstruction Theorem]
\label{thm:CanonicalObstruction}
If for a Thurston map with hyperbolic orbifold, its canonical obstruction is empty then it is Thurston equivalent to a rational function. If the canonical obstruction is not empty then it is a Thurston obstruction.
\end{theorem}

\begin{theorem}[Curves Degenerate or Stay Bounded]
\label{thm:curvebound}
For any point $\tau \in \T_f$ there exists a bound $L=L(\tau, f)>0$ such that for any essential simple closed curve $\gamma \notin \Gamma_f$ the inequality $l(\gamma,\sigma_f^n(\tau)) \ge L$ holds for all $n$.  
\end{theorem}

Recall that the augmented \Teich space $\AT(S)$ is the space of all stable  Riemann surfaces with marked points \emph{with nodes} of the same type as $S$ (see, for example, \cite{W1,W2}). The \emph{type} of a nodal surface is defined by its topological type (more precisely, by the topological type  of a  surface one obtains by opening up all nodes) and the number of marked points (excluding nodes).  The augmented \Teich space $\AT_f$ is a stratified space with strata corresponding to multicurves on $(\Sphere,Q_f)$. We denote by $\S_\Gamma$ the stratum corresponding to the multicurve $\Gamma$, i.e., the set of all nodal surfaces for which the nodes come from pinching all elements of $\Gamma$ and for which there are no other nodes. In particular, $\T_f = \S_\emptyset$.  Strata of $\AM_f$ are labeled by equivalence classes $[\Gamma]$ of multicurves, where two multicurves $\Gamma_1$ and $\Gamma_2$ are in the same class if and only if one can be transformed to the other by an element of the pure mapping class group or, equivalently, if the respective elements of $\Gamma_1$ and $\Gamma_2$ separate points of $Q_f$ in the same way.

In \cite{S11}, the author proves the following statement.

\begin{theorem} 
\label{thm:extensionA} Thurston's pullback map extends continuously to a self-map of the augmented \Teich space.
\end{theorem}

\begin{remark} Theorems \ref{thm:CanonicalObstruction} and \ref{thm:curvebound} can be reformulated as follows. Note that this version does not require that the Thurston map $f$ have hyperbolic orbifold.
\end{remark}
\begin{theorem}
\label{thm:pilgrims}
The accumulation set of $\pi(\sigma_f^n(\tau))$ in the compactified moduli space $\AM_f$  is a compact subset of  $\S_{[\Gamma_f]}$. 
\end{theorem}

We will need the following proposition (compare  \cite[Lemma 5.2]{DH}).
\begin{proposition}
\label{prop:covers}
 There exists an intermediate cover $\M'_f$ of $\M_f$ (so that
  $\T_f \stackrel{\pi_1}{\longrightarrow} \M'_f \stackrel{\pi_2}{\longrightarrow} \M_f$ 
 are covers and $\pi_2\circ\pi_1=\pi$) such that
 \begin{itemize}
	\item[i.] $\pi_2$ is finite,
	\item[ii.]  the diagram
\begin{equation}
\label{diag:covers}
\begin{diagram}
\node{\T_f} \arrow[2]{e,t}{\sigma_f} \arrow{se,l}{\pi_1} \arrow[2]{s,l}{\pi} \node[2]{\T_f}
\arrow[2]{s,l}{\pi}
\\
\node[2]{\M'_f} \arrow{se,t}{\sigmat_f} \arrow{sw,t}{\pi_2}
\\
\node{\M_f}  \node[2]{\M_f}
\end{diagram}
\end{equation}
 commutes for some map ${\sigmat_f} \colon \M'_f \to
 \M_f$,
 \item[iii.] If  $\pi_1(\tau_1)=\pi_1(\tau_2)$ then $f_{\tau_1}=f_{\tau_2}$ up to pre- and post-composition by Moebius transformations.

\end{itemize}

\end{proposition}

Note that $\M'_f$ is a quotient of $\T_f$ by a subgroup $G$ of the pure mapping class group of finite index. Then the quotient $\AM'_f$ of $\AT_f$ by the same subgroup will be a compactification of $\M'_f$. The covers $\pi_1$ and $\pi_2$ can be extended to  the corresponding augmented spaces so that $\pi_2\circ\pi_1=\pi$ still holds. As in the case of the compactified moduli space, we parametrize boundary strata by equivalence classes of multicurves, two classes of simpled closed curves being equivalent if one can be mapped to the other by the action of $G$. Evidently, the whole diagram above also extends to the augmented spaces.

\section{Thurston maps with parabolic orbifolds}
\label{sec:parabolic}

A complete classification of postcritically finite branched covers (i.e. Thurston maps $(f,Q_f)$, where $Q_f$ is equal to the postcritical set $\PP$) with parabolic orbifolds has been given in \cite{DH}. All rational functions that are postcritically finite branched covers with parabolic orbifold has been extensively described in \cite{milnorlattes}. However, no classification has been developed yet for general Thurston maps. In this section, we remind the reader of basic results on Thurston maps with parabolic orbifolds.

Recall that a map $f \colon (S_1,v_1) \to (S_2,v_2)$ is a covering map of orbifolds if $v_1(x)\deg_x f = v_2(f(x))$ for any $x \in S_1$. The following proposition from \cite{DH} is crucial.

\begin{proposition} 
\begin{itemize}
	\item[i.] 
If $f \colon \Sphere \to \Sphere$ is a postcritically finite branched cover, then $\chi(O_f) \le 0$.
\item[ii.] If $\chi(O_f) = 0$, then $f \colon O_f \to O_f$ is a covering map of orbifolds.
\end{itemize}
\end{proposition}

Equation~(\ref{eq:euler}) gives six possibilities for $\chi(O_f)=0$. If we record all the values of $v_f$ that are bigger than 1, we get one of the following orbifold signatures.

\begin{itemize}
	\item $(\infty,\infty),$
	\item $(2,2,\infty),$
	\item $(2,4,4),$
	\item $(2,3,6),$
	\item $(3,3,3),$
	\item $(2,2,2,2).$
\end{itemize}

We are mostly interested in the last case. A \emph{$(2,2,2,2)$-map} is a Thurston map that has orbifold with signature $(2,2,2,2)$. From now on in this section, we always assume that $f$ is a $(2,2,2,2)$-map. An orbifold with signature $(2,2,2,2)$ is a quotient of a torus $T$ by an involution $i$; the four fixed points of the involution $i$ correspond to the points with ramification weight 2 on the orbifold. Denote by $p$ the corresponding branched cover from $T$ to $\Sphere$; it has exactly 4 simple critical points which are the fixed points of $i$. It follows that $f$ can be lifted to a covering self-map $\fh$ of $T$ (see \cite{DH}). For a   $(2,2,2,2)$-map $f=(f,Q_f)$, we will denote by $F=(f,P_f)$ the corresponding $(2,2,2,2)$-map with exactly 4 postcritical points in the marked set.

Take any simple closed  curve $\gamma$ on $\Sphere \setminus \PP$. Then $p^{-1}(\gamma)$ has either one or two components that are simple closed curves. 

\begin{proposition}
 If there are exactly two postcritical points of $f$ in each complementary component of $\gamma$, then the $p$-preimage of $\gamma$ consists of two components that are homotopic in $T$ and non-trivial in $H_1(T,\Z)$. Otherwise, all preimages of $\gamma$ are trivial.
 \label{prop:trivialhom}
\end{proposition}
\begin{proof} Note that postcritical points of $f$ are, by definition, the fixed points of involution $i$ and critical values of $p$. Since $\gamma$ separates $\Sphere$ into two connected components, if $\gamma$ has exactly one preimage $\alpha$, then $\alpha$ must separate $T$ into two components and, thus, be contractible. The disk bounded by $\alpha$ in $T$ is mapped to one of the connected components of the complement of $\gamma$ with degree two. Therefore, there exists exactly one critical point of $p$ is this disk, and, hence, exactly one postcritical point in its image.

If $\gamma$ has two preimages, then they do not intersect, which implies that they are homotopic to each other or at least one of them is contractible. However, if one of the preimages is contractible, the other one also is, since they are mapped to each other by $i$. In this case, the disk bounded by a contractible preimage component is mapped by $f$ one-to-one to a component of the complement of $\gamma$. Thus, this complementary component has no postcritical points.

The last case is when $\gamma$ has two preimages that are homotopic to each other and are not trivial in $H_1(T,\Z)$. Then the complement of the full preimage of $\gamma$  has two components that are annuli that are mapped by $p$ to complementary components of $\gamma$ with degree two. It follows that there are exactly two postcritical points in each of these components.
\end{proof}

Every homotopy class of simple closed curves $\alpha$ on $T$ defines, up to sign, an element $\langle \alpha \rangle$ of $H_1(T,\Z)$. If a simple closed curve $\gamma$ on $\Sphere \setminus \PP$ has two $p$-preimages, then they are homotopic by the previous proposition. Therefore, every  homotopy class of simple closed curves $\gamma$ on $\Sphere \setminus \PP$ also defines, up to sign, an element $\langle \gamma \rangle$ of $H_1(T,\Z)$. It is clear that for any $h \in H_1(T,\Z)$ there exists a homotopy class of simple closed curves $\gamma$ such that $h=n \langle \gamma \rangle$ for some $n \in \Z$.

Since $H_1(T,\Z)\cong \Z^2$, the push-forward operator $\fh_*$ is a linear operator. It is easy to see that the determinant of $\fh_*$ is equal to the degree of $\fh$, which is in turn equal to the degree of $f$. Existence of invariant multicurves for $f$ is related to the action of $\fh_*$ on $H_1(T,\Z)$. 

\begin{proposition} 
\label{prop:homology1} Suppose that a component $\gamma'$ of the $f$-preimage of a simple closed curve $\gamma$ on $\Sphere \setminus \PP$  is homotopic to $\gamma$ relative to $\PP$. Take a $p$-preimage $\alpha$ of $\gamma$. Then $\fh_*(\langle\alpha\rangle)=\pm d \langle\alpha\rangle$, where $d$ is the degree of $f$ restricted to  $\gamma'$.
\end{proposition}

\begin{proof} By the previous proposition, if $\gamma$ is not essential in $\Sphere \sm \PP$ then $\langle\alpha\rangle=0$ and the claim holds. Otherwise, both $\gamma$ and $\gamma'$ have exactly two components in their $p$-preimages. Denote by $\alpha'$ a component of the $p$-preimage of $\gamma'$ that is mapped to $\alpha$ by $\fh$ with degree $d$. A homotopy between $\gamma$ and $\gamma'$ lifts to a homotopy between $\alpha$ and $\alpha'$ on $T$. Therefore, $\alpha$ and $\alpha'$ define the same (up to sign) element of $H_1(T,\Z)$, i.e. $\langle\alpha'\rangle=\pm \langle\alpha\rangle$. Since $\fh(\alpha')=\alpha$, we get that $\fh_*(\langle\alpha\rangle)=\fh_*(\pm\langle\alpha'\rangle)=\pm d \langle\alpha\rangle$. In particular, this implies that all connected components of the preimage of $\gamma$ map to $\gamma$ with the same degree. 
\end{proof}

More generally, we obtain the following.

\begin{proposition} 
\label{prop:homology} Let $\gamma$ be a simple closed curve  on $\Sphere \setminus \PP$ such that there are two points of the postcritical set $\PP$ in each complementary component of $\gamma$. If all components of the  $f$-preimage of $\gamma$ have zero intersection number  with $\gamma$ in $\Sphere \setminus \PP$, then $\fh_*(\langle\gamma\rangle)=\pm d \langle\gamma\rangle$, where $d$ is the degree of $f$ restricted to any preimage of  $\gamma$.
\end{proposition}

\begin{proof} Let $\gamma'$ be a component of the $f$-preimage of $\gamma$. As before, denote by $\alpha'$ a component of the $p$-preimage of $\gamma'$ that is mapped  by $\fh$ to a component $\alpha$ of the $p$-preimage of $\gamma$ with some degree $d$. Then
$\fh_*(\langle\alpha'\rangle)=\pm d \langle\alpha\rangle$. Since $\gamma$ and $\gamma'$ have zero intersection number  in $\Sphere \setminus \PP$, the homotopy classes of $\alpha$ and $\alpha'$ in $T$ also have zero intersection number. By Proposition~\ref{prop:trivialhom}, $\langle\alpha\rangle$ is non-zero in $H_1(T,\Z)$; the last equality implies that $\langle\alpha'\rangle$ is also non-zero in $H_1(T,\Z)$. It follows that $\langle\alpha'\rangle=\pm \langle\alpha\rangle$.
\end{proof}

\section{Topological characterization of canonical obstructions}
\label{sec:canonical}

Recall that the canonical obstruction of a Thurston map $f$ is defined as the set of all homotopy classes of curves for which the length of the corresponding geodesics tends to 0 as we iterate in the \Teich space $T_f$. Notice that this definition makes sense in the case of Thurston maps with parabolic orbifolds. However, only one of the implications in Theorem~\ref{thm:CanonicalObstruction} is true in this setting.

\begin{proposition}
 If a Thurston map $f$ with a parabolic orbifold is Thurston equivalent to a rational map then its canonical obstruction $\Gamma_f$ is empty.
\end{proposition}

\begin{proof}
If we start iterating $\sigma_f$ at a fixed point $\tau$ then, obviously, the lengths of all geodesics are uniformly bounded from below.
\end{proof}

The other direction of Theorem~\ref{thm:CanonicalObstruction} tells that if for a Thurston map with hyperbolic orbifold the canonical obstruction is empty then there exist no obstructions at all for this map. In the general case, the following is true.

\begin{theorem} Suppose that  the canonical obstruction of  a Thurston map $f$ is empty, and $\Gamma$ is a simple Thurston obstruction for $f$. Then $f$ is a $(2,2,2,2)$-map and every curve of $\Gamma$ has two postcritical points of $f$ in each complementary component.
\label{thm:noncanonical}
\end{theorem}

\begin{proof}The map $f$ must have parabolic orbifold by Theorem~\ref{thm:CanonicalObstruction}.  Since the canonical obstruction is empty, the leading eigenvalue of $M_\Gamma$ is equal to 1. The proof is similar to McMullen's classification of multicurves $\Gamma$ with $\lambda(\Gamma)=1$ for rational maps (compare \cite{McMbook}).

{\bf Case I.} As an illustration of the idea of the proof, we first consider the case when $\Gamma$ consists of a single simple closed curve $\gamma$.  Then the Thurston matrix has the form $M_\Gamma=(1)$. Denote by $r(\tau)=M(\gamma,\tau)$ the maximal modulus of an annulus homotopic to $\gamma$ in the Riemann surface corresponding to $\tau$. (We can always find an annulus of maximal modulus by Lemma~\ref{lem:modulus}.) Recall, that  by the Collar Lemma, this is approximately equal to $1/l(\gamma,\tau)$. 
Then, by the Gr\"otzsch inequality (see more detailed explanation below), $r(\sigma_f(\tau)) \ge r(\tau)$. As usual, denote $\tau_n =\sigma^n_f(\tau_0)$. The sequence $r_n=r(\tau_n)$ is non-decreasing and bounded because $\gamma$ is not a part of canonical obstruction and, therefore, $l(\gamma,\tau_n)>L>0$ for some $L$. It follows that $r_n$ has a limit which we denote by $r$.

We pick initial $\tau_0$ so that  $r_0>\pi/\ln(3+2\sqrt{2})$ to make sure that the geodesic corresponding to $\gamma$ is shorter than any simple closed geodesic that intersects $\gamma$ by the Collar Lemma. 

 Theorem~\ref{thm:pilgrims} implies that all $m'_n=\pi_1(\tau_n)$ belong to a compact subset of $\M'_f$, where $\pi_1$ is defined as before (see Proposition~\ref{prop:covers}). Consider a subsequence $n_k$ such that $\{m'_{n_k}\}$ converges to $p \in \M'_f$ and  $\{m'_{n_k+1}\}$ converges to $q \in \M'_f$. On either of  Riemann surfaces corresponding to $p$ and $q$, there exists exactly one homotopy class of simple closed curves $\gamma'$ such that $\gamma'$ and $\gamma$ are in the same equivalence class under the action of the covering group corresponding to $\pi_1$, and that $M(\gamma',p)=r$ (or $M(\gamma',q)=r$ respectively). Indeed, any two different curves in the same equivalence class have non-zero intersection number because otherwise they would bound an annulus that cannot contain any of the marked points and, thus, be homotopic to each other.  Since $r>r_0$ is large enough, $M(\gamma'',p)<r$  (and the same for $q$) for any other curve $\gamma''$ in the equivalence class of $\gamma$ because otherwise $\gamma'$ and $\gamma''$ would not intersect by the Collar Lemma. The same statement obviously holds for equivalence classes of curves with respect to the pure mapping class group.

Recall the commutative diagram (\ref{diag:covers}):
\begin{equation*}
\begin{diagram}
\node{\T_f} \arrow{e,t}{\sigma_f} \arrow{s,l}{\pi_1} \node{\T_f}
\arrow{s,l}{\pi}
\\
\node{\M'_f} \arrow{e,t}{\sigmat_f} \node{\M_f}
\end{diagram}
\end{equation*}

Since all the maps involved are continuous, it follows that $\sigmat_f(p)= \lim \sigmat_f(m'_{n_k})=\lim \pi(\tau_{n_k+1})=\pi_2(q)$. Take any point $\tau \in \T_f$ such that $\pi_1(\tau)=p$ and $r(\tau)=M(\gamma,\tau)=r$ (a point $p$ provides a conformal structure on the base topological surface together with marking of some equivalence classes of curves, to get a  point in the fiber $\pi_1^{-1}(p)$ we simply choose a representative of each marked class; for the equivalence class of $\gamma$ we choose the shortest representative). Then from the same diagram we get $\pi(\sigma_f(\tau))=\pi_2(q)$. Therefore $r(\sigma_f(\tau))$ is bounded above by $M(\gamma',\pi_2(q))$ where $\gamma'$ is equivalent to $\gamma$ under the action of the pure mapping class group. As we have seen above, this implies that $r(\sigma_f(\tau)) \le r$. On the other hand, we know that $r(\sigma_f(\tau)) \ge r(\tau)= r$ which implies $r(\sigma_f(\tau))=r(\tau)=r$. We now investigate under which conditions the inequality $r(\sigma_f(\tau)) \ge r(\tau)$ can become an equality. 

Denote by $A$ an annulus in the Riemann surface corresponding to $\tau$ that is homotopic to $\gamma$ relative to $Q_f$ and  has maximal possible modulus $r$ (see Lemma~\ref{lem:modulus}); let annuli $A_1, \ldots, A_k$ be the disjoint preimages of $A$ in the Riemann surface corresponding to $\sigma_f(\tau)$ under the map $f_\tau$ that are homotopic to $\gamma$. 
 Each $A_i$ is mapped to $A$ by $f_\tau$ as a non-ramified cover of degree $d_i$. Therefore, $\mod A_i = 1/d_i \, \mod A = r/d_i$.  By the definition of Thurston obstructions, $M_{\{\gamma\}} = (\sum 1/d_i)$ hence $\sum 1/d_i=1$.  Consider the smallest annulus $B$ containing all $A_i$ and homotopic to $\gamma$ relative to $Q_f$. By the Gr\"otzsch inequality, $\mod B \ge \sum \mod A_i = \sum (1/d_i)r =r$. 

If we assume that $r(\sigma_f(\tau)) = r(\tau)$ then $\mod B \le r$, hence $B$ is an annulus homotopic to $\gamma$ relative to $Q_f$ with  maximal possible modulus and the last inequality is, in fact, an equality. Hence, the closure of $B$ must be the whole Riemann sphere; and $A_i$ are round subannuli of $B$ and their closure covers $B$. If there were only one preimage $A_1$ of $A$, then the degree $d_1$ would have to be equal to 1; this would imply that $f$ is not a Thurston map but a homeomorphism because the closure of $B=A_1$ is the whole sphere. Every two adjacent annuli share a common boundary component $\alpha$ which is an analytic curve. Then the corresponding boundary component $\beta=f_\tau(\alpha)$ of $B$ is piece-wise (except for possibly at the images of critical points) smooth. Since both adjacent to $\alpha$ annuli   are mapped to $A$, 
$\beta$ has only one complementary component. 
We conclude that $\beta$ is a smooth curve segment connecting two critical values of $f$ and passing through no other critical value, because a component of the preimage  of $\beta$ is a simple closed curve. Moreover, all critical points of $f_\tau$ in $\alpha$ are of degree 2.
If the number of annuli $k$ is at least 3 then the same is true for the other boundary component $\delta$ of $A$. Hence there exist at least 4 critical values -- at least two in each complementary component of $A$ --  and this yields the statement of the theorem. 

We can reduce the case $k=2$ to the case, say, $k=4$ by considering the second iterate of $f$. Alternatively, we can prove the theorem in the case $k=2$ directly as follows.

The number of critical points on $\alpha$ is equal to the degree of $f$ on $\alpha$, which is twice the degree of $f$ on either of the annuli; in particular, this yields $d_i$ are all equal. Thus, in the case of exactly two annuli, the map $f$ has degree 4. We already know that there are exactly 4 critical points of $f_\tau$ on $\alpha$. Each of the two preimages $\delta_1$ and $\delta_2$ of $\delta$ must contain at least one critical point because otherwise $A_1 \cup \delta_1$ (or $A_2 \cup \delta_2$ respectively)  contains no critical points and, hence, maps by $f_\tau$ conformally on $A \cup \delta$ which is a contradiction. Therefore, $f$ has exactly 6 simple critical points since a branched cover of degree 4 has at most $4*2-2=6$ critical points. This implies that $f$ is a $(2,2,2,2)$-map. Indeed, the degree of $f$ is 4 so every critical value is the image of at most two different critical points. Thus, we must have at least 3 critical values and their full preimage consists of the 6 critical points. If the postcritical set contains only 3 points, then all critical values are also critical points and the signature of the orbifold corresponding to $f$ is $(\infty,\infty,\infty)$, which contradicts the fact that $f$ has parabolic orbifold. We see that $\PP$ contains at least 4 points; since $f$ has parabolic orbifold we conclude that $f$ is a $(2,2,2,2)$-map.


Note that the endpoints of $\beta$ are the only two postcritical points of $f$ on $\beta$, 
 because every other point on $\beta$ has exactly 4 preimages on $\alpha$ and none of them are either critical or marked. Therefore, each complementary component of $\gamma$ has two postcritical points.

{\bf Case II.} We use the same approach in the case when $M_\Gamma=\{\gamma_1,\ldots,\gamma_k\}$ is positive and $k \ge 2$. 
 By the Perron-Frobenius theorem (Theorem~\ref{thm:p-f}), there exists a positive eigenvector $v$ corresponding to $\lambda_\Gamma=1$. Denote $r(\tau)=\sup \min_{i=\overline{1,k}} \mod A_i/v_i$, where the supremum is taken over all configurations of disjoint annuli $A_i$ in the Riemann surface corresponding to $\tau$, such that $A_i$ is homotopic to $\gamma_i$ for $i=\overline{1,k}$.

Consider annuli $B_i$ in the Riemann surface corresponding to $\sigma_f(\tau)$, such that  $B_i$ is homotopic to $\gamma_i$ and $B_i$ contains all preimages of $A_j$ that are homotopic to $\gamma_i$ for every $i=\overline{1,k}$. The same arguments as above show that $(\mod B_i)^T \ge M_\Gamma (\mod A_i)^T$. If $A_i$ is the maximal configuration that realizes $r(\tau)$ (again, existence of such a configuration follows from Lemma~\ref{lem:modulus}), then, by definition, $(\mod A_i)^T \ge r(\tau)v$ and, hence, $(\mod B_i)^T \ge M_\Gamma r(\tau)v = r(\tau)v$. Thus, $r(\sigma_f(\tau)) \ge r(\tau)$. We use pre-compactness of $m'_n$ as above to construct a point $\tau$ such that $r(\sigma_f^2(\tau))=r(\sigma_f(\tau)) = r(\tau)=r$. This means that there exists $i$ for which $\mod B_i=r v_i$ and the inequality corresponding to the $i$-th line of  $(\mod B_i)^T \ge M_\Gamma (\mod A_i)^T$ is an equality. Since all entries of $M_\Gamma$ are positive, this immediately forces $\mod A_i=r v_i$ for all $i$. Because $r(\sigma_f^2(\tau))=r$, the same reasoning implies that $\mod B_i=r v_i$, for all $i$, and $\{B_i\}$ is a maximal confi\-guration for $\sigma_f(\tau)$. In particular, we infer that the closure of the union of $B_i$ covers the whole sphere, and $\Gamma$ is invariant and no curve of $\Gamma$ has a non-essential preimage.

Moreover, all  $f_\tau$-preimages of annuli $A_j$ that are homotopic to $\gamma_i$ are round subannuli of $B_i$ and  closure of their union covers $B_i$. If two preimages of some $A_j$ abut along a boundary component then the image boundary component of $A_j$ is a smooth curve segment (see above). If preimages of two different $A_j$ and $A_l$ abut in $\sigma_f(\tau)$ then the annuli themselves abut in $\tau$. Since for every $j$, there exists at least one preimage of $A_j$ that is homotopic to $\gamma_i$, all annuli are concentric and so are the curves $\gamma_i$. Since $B_i$ are homotopic to $\gamma_i$, it follows that all preimages of $A_i$ are also concentric. We see that no boundary component that separates two annuli $A_i$ can contain a postcritical point. By removing marked points on these boundary components, we can reduce the statement of the theorem to the previous case.

{\bf Case III.} The general case is now easily reduced to the previous case. Suppose $\Gamma$ is an arbitrary simple obstruction. Take a subset $\Gamma_1$ such that $M_{\Gamma_1}$ is irreducible and $\lambda_{\Gamma_1}=1$ (see Section~\ref{sec:matrices}). By Theorem~\ref{thm:imprimitive}, some iterate $k$ of $M_{\Gamma_1}$ can be conjugated by a permutation matrix  into the block form where all the blocks on diagonal are positive and all other blocks are zero, moreover, the leading eigenvalue of each block is 1. Take further subset $\Gamma_2 \subset \Gamma_1$ corresponding to any of the diagonal blocks. Then $M_{\Gamma_2}$ with respect to $f^k$  is a positive matrix with leading eigenvalue 1. From the previous cases, we know that $f^k$ is a $(2,2,2,2)$-map and that every element of $\Gamma_2$ has two postcritical points in each complementary component. This immediately implies, that $f$ itself is a  $(2,2,2,2)$-map with the same postcritical set. Moreover, we get that $\Gamma_2$ is invariant. Therefore, $\Gamma_1$ is a union of invariant multicurves for $f^k$  and no curve of these multicurves has a non-essential preimage. It follows that $\Gamma_1$ is $f$-invariant and all curves in $\Gamma_1$ have two postcritical points in each complementary component.

Set $\Gamma'=\Gamma \setminus \Gamma_1$. Then $\lambda_{\Gamma'}=1$ because $\Gamma$ is simple and $\Gamma_1$ is invariant. We repeat the argument until we exaust $\Gamma$, proving the desired property for all curves in $\Gamma$. 
\end{proof}

Recall that to a $(2,2,2,2)$-map $f$ we associate a linear operator $\fh_* \colon H_1(T,\Z) \to H_1(T,\Z)$ (see Section~\ref{sec:parabolic}). The following statement is in a certain sense a converse to the previous theorem.

\begin{theorem}
\label{thm:2222obstruction}
The canonical obstruction of a $(2,2,2,2)$-map $f$ contains a curve that has two postcritical points in each complementary component if and only if $\fh_*$ has two different integer eigenvalues.
\end{theorem}

\begin{proof}
Let $\fh_*$ have two different integer eigenvalues $d_1$ and $d_2$.Without loss of generality, we assume that both are positive and $d_1<d_2$. We first assume that $Q_f$ has exactly four points so that $\PP=Q_f$.  Take a curve $\gamma$ in $\Sphere \setminus Q_f$ such that $\langle \gamma \rangle$ is an eigenvector of $\fh_*$ corresponding to $d_1$. For any connected component $\gamma'$ of the $f$-preimage of $\gamma$, we have $\fh_*(\langle \gamma' \rangle)=d\langle \gamma \rangle$ with some $d\in\Z$. It follows that $\langle \gamma' \rangle$ is also an eigenvector of $\fh_*$ corresponding to $d_1$ because $\fh_*$ is diagonalizable. We infer that all preimages of $\gamma$ are homotopic to  $\gamma$ and each of them is mapped to $\gamma$ by $f$ with degree $d_1$. Since the degree of $f$ is equal to $d_1d_2$, there are exactly $d_2$ preimages of $\gamma$. Therefore, $\Gamma=\{\gamma\}$ is an obstruction with $M_\Gamma=(d_2/d_1)$ and $\Gamma \subset \Gamma_f$ because the leading eigenvalue $\lambda_\Gamma=d_2/d_1>1$.

If $Q_f$ has more then four points, we first consider the Thurston map $F=(f, \PP)$ where $\PP$ is the postcritical set of $f$. By the arguments above, there exists a curve $\gamma$ that has two postcritical points in each complementary component which is a part of the canonical obstruction of $F$. Clearly, there exists a canonical projection $p$ from $\T_f$ to $\T_F$ that makes the following diagram commute: 
\[
\begin{diagram}
\node{\T_f} \arrow{e,t}{\sigma_f} \arrow{s,l}{p} \node{\T_f}
\arrow{s,l}{p}
\\
\node{\T_F} \arrow{e,t}{\sigma_F} \node{\T_F}
\end{diagram}
\]

As we iterate $p(\tau_1)$ in the \Teich space $\T_F$, the maximal modulus $M(\gamma,p(\tau_m))$ of an annulus homotopic to $\gamma$ tends to infinity. This modulus is the same as the maximal modulus on an annulus $A$ homotopic to $\gamma$ in $\tau_m \in T_f$ if we fill in all extra punctures. The extra marked points split $A$ into at most $k+1$ concentric annuli, where $k$ is the number of points in $Q_f\setminus \PP$. Thus, the lengths of all homotopy classes of curves in $\Sphere \setminus Q_f$ that are homotopic to $\gamma$ relative to $\PP$ cannot be uniformly bounded from below as we iterate. Indeed, in this case all concentric annuli must have modulus  uniformly bounded from above, and since there are at most $k+1$ of them, the modulus of an annulus homotopic to $\gamma$ relative to $\PP$ would be also bounded from above, which is a contradiction. By Theorem~\ref{thm:pilgrims}, there exists a curve homotopic to $\gamma$ relative to $P_f$ in the canonical obstruction $\Gamma_f$.


Let us now prove the converse direction; suppose that $\gamma$ is a curve in the canonical obstruction of $f$ that has two postcritical points in each complementary component. This immediately implies that $\gamma$ is also an element of the canonical obstruction of $F=(f,\PP)$ since any annulus in $\Sphere \setminus Q_f$ is also an annulus in $\Sphere \setminus \PP$. We suppose, by contradiction, that both  eigenvalues of $\fh_*$ are equal to some $d\in \N$. If we fix an appropriate basis of $H_1(T,\Z)$, the operator $\fh_*$ assumes the form
$$  \fh_*= 
\left( \begin{array}{cc}
d & a \\	
0 & d
\end{array} \right),
$$
where $a\in\Z$. After identifying $\T_F$ with the hyperbolic plane $\mathbb{H}$,  Thurston's pullback assumes the form $\sigma_F(\tau)=\tau +a/d$ (see Lemma~9.6 in \cite{DH}). We see that the orbits of this map are bounded in the moduli space $\M_F$, which yields that the canonical obstruction of $F$ must be empty.
\end{proof}

\begin{corollary}
\label{cor:emptyobstruction}
The canonical obstruction of a $(2,2,2,2)$-map $f$ is empty if and only if every curve of every simple Thurston obstruction for $f$ has two postcritical points of $f$ in each complementary component and the two eigenvalues of $\fh_*$ are equal or non-integer. For any other Thurston map, the canonical obstruction is empty if and only if there exist no Thurston obstruction for $f$.
\end{corollary}

\begin{proof}
The statement follows directly from the previous two theorems.
\end{proof}

 Recall that we can decompose $f$ along a completely invariant multicurve $\Gamma$ in the following manner. We consider two copies of the sphere $(\Sphere, Q_f)$ such that $f$ maps the first copy onto the second copy. We pinch to a single point every curve of $\Gamma$ in the range sphere, and we pinch to a single point every curve of $f^{-1}(\Gamma)$ in the domain sphere. Now $f$ is a branched cover between two nodal spheres, i.e. a collection of branched covers sending components of the domain onto components of the range; all nodes are considered to be marked points for appropriate covers. If we perform the stabilization process on the domain nodal sphere, which is done by further pinching to a single point every component of the domain that has at most 2 marked points (including nodes), we obtain a nodal surface that is canonically homeomorphic (up to isotopy relative to marked points) to the nodal sphere in the range. Every component of this nodal sphere is pre-periodic and the dynamics of $f$ splits into dynamics on periodic components. The latter can be understood by studying the first return map $F$ of a given periodic component $C$ taking  $Q_F$ to be the set of marked points and nodes on this component. This map is either a Thurston map or a homeomorphism.
 
 The following theorem is a generalization of Pilgrim's conjecture (\cite[Theorem~10.3]{S11}). We use a different approach here.

\begin{theorem}
\label{thm:KevinConjecture2}
  If the first-return map $F$ of a periodic component $C$ of the nodal topological surface corresponding to pinching curves in $\Gamma_f$ is a Thurston map (i.e. the degree of $F$ is greater than 1), then the canonical obstruction of  $F$ is empty.
\end{theorem}

\begin{proof}

The main idea of the proof is essentially the same as in the proof of Theorem~\ref{thm:noncanonical}. Suppose, on contrary, that $\Gamma\neq\emptyset$ is the canonical obstruction of $F$. Passing to an appropriate iterate of $f$, 
we may assume that $C$ is a fixed component (see Proposition~\ref{prop:iterate}). In this case, $F$ is simply the restriction of $f$ to the component $C$. 

The multicurve $\Gamma$ can also be viewed as a multicurve in $\Sphere \setminus Q_f$. Since $\Gamma$ is not a part of the canonical obstruction of $f$ (the curves in $\Gamma$ live in a component obtained after pinching all curves in $\Gamma_f$), the leading eigenvalue $\lambda_\Gamma$ is equal to 1. Taking a higher iterate of $f$, if needed, we can find a subset $\Gamma' \subset \Gamma$ such that $\lambda_{\Gamma'}=1$ and $M_{\Gamma'}$ is positive.

We can repeat the proof of Case II of Theorem~\ref{thm:noncanonical} almost verbatim. As before, there exists a positive eigenvector $v$ corresponding to the leading eigenvalue 1. Denote $r(\tau)=\sup \min_{i=\overline{1,k}} \mod A_i/v_i$, where the supremum is taken over all configurations of disjoint annuli $A_i$ in the Riemann surface corresponding to $\tau$, such that $A_i$ is homotopic to $\gamma_i \in \Gamma'$. We have already established that $r(\sigma_f(\tau)) \ge r(\tau)$. Take an accumulation point $m'$ of the projection of an arbitrary orbit to the space $\M'$. Then $m' \in \S_{\Gamma_f}$ by Theorem~\ref{thm:pilgrims} and $r(m')=r(\sigmat_f(m'))$. Since all annuli that are homotopic to curves in $\Gamma'$  on 
any nodal surface in $\S_{\Gamma_f}$ must be contained in $C$, from this point on the proof goes the same way as the proof of  Case II of Theorem~\ref{thm:noncanonical}. We conclude that $F$ is a $(2,2,2,2)$-map and all curves of $\Gamma'$ have two postcritical points of $F$ in each complementary component. However, by the previous theorem, no such curve can be contained in the canonical obstruction of $F$. Otherwise, the leading eigenvalue of $\Gamma'$ would be strictly greater than 1, which would force $\Gamma'$ to be a part of the canonical obstruction of $f$. This contradiction shows that the canonical obstruction of $F$ is empty.
\end{proof}

We are now in the position to formulate and prove a pure topological criterion that singles out canonical obstructions. It says that the canonical obstruction is the minimal obstruction satisfying the conclusion of the previous theorem.

\begin{theorem}[Characterization of Canonical Thurston Obstructions]
\label{thm:main}
 The canonical obstruction $\Gamma$ 
is a unique minimal Thurston obstruction with the following properties. 
  	\begin{itemize}
		\item If the  first-return map $F$ of a cycle of components  in $\S_\Gamma$ is a $(2,2,2,2)$-map, then  every curve of every simple Thurston obstruction for $F$ has two postcritical points of $f$ in each complementary component and the two eigenvalues of $\hat{F}_*$ are equal or non-integer.
		\item If the first-return map $F$ of   a cycle of components  in $\S_\Gamma$ is not a $(2,2,2,2)$-map or a homeomorphism, then there exists no Thurston obstruction of $F$.
	\end{itemize}
\end{theorem}

\begin{proof}
By Corollary~\ref{cor:emptyobstruction}, both conditions above are equivalent to saying that the canonical obstruction for $F$ is empty. The necessity of these conditions then follows from Theorem~\ref{thm:KevinConjecture2}. Suppose, on contrary, that $\Gamma$ is not minimal with these properties, i.e. there exists $\Gamma' \subsetneq \Gamma$ satisfying the same condition. Since $\Gamma$ is simple by Proposition~\ref{prop:simple}, at least one curve $\gamma$ of $\Gamma \setminus \Gamma'$ must lie in a periodic component $C$ of $\S_{\Gamma'}$. Consider the multicurve  $\Gamma_1$ containing all curves of $\Gamma \setminus \Gamma'$ that lie in $C$; as $\Gamma$ is simple, $\Gamma_1$ is an obstruction for $F$. There are three cases.

{\bf Case I.} The first-return map to $C$ is a homeomorphism. Since $\gamma$ is essential and not homotopic to any nodes, the component $C$ must have at least four marked points. Take any other simple closed curve $\alpha_1$ in $C$ that has non-zero intersection with $\gamma$. Since  the first-return map is a homeomorphism, $\gamma$ is a part of a Levy cycle. Denote by $\alpha_2$ the one-to-one preimage of $\alpha_1$ and so on. Then, as in the proof of Theorem~\ref{thm:2222obstruction}, we see that  $M(\alpha_{i+1},\tau_{i+1})\ge M(\alpha_i,\tau_i)$ and $\alpha_{ki}$ intersects $\gamma$ for all $i$, where $k$ is the length of the Levy cycle. This implies that the length of $\gamma$ is bounded from above, which contradicts the assumption that $\gamma$ was a part of the canonical obstruction.

{\bf Case II.} The first-return map to $C$ is a $(2,2,2,2)$-map. It follows that every curve in $\Gamma_1$ has two postcritical points of $F$ in each complementary component. By the same argument as in the proof of Theorem~\ref{thm:2222obstruction}, the length of these curves cannot tend to zero as we iterate $f$ so $\Gamma_1$ cannot be a part of the canonical obstruction.

{\bf Case III.} The first-return map to $C$ is neither of the two cases above. By Corollary~\ref{cor:emptyobstruction}, $F$ has no obstructions at all, hence $\Gamma_1$ must be empty.

Since any curve of any other obstruction either lies in $\Gamma$, or does not intersect any curve in $\Gamma$, the uniqueness of a minimal obstruction satisfying the conditions of the theorem follows from the same argument. 
\end{proof}

\appendix
\section{Continuity of modulus}

In this section we prove the following lemma (compare to \cite[Section I.4.9]{LV}).

\begin{lemma}
\label{lem:modulus} Let $A_i$ be an annulus in $\P$ with two complementary components $B_i$ and $C_i$, for every $i\in\N$. Suppose that $B_i \to B$ and $C_i \to C$ as $i \to \infty$ with respect to the Hausdorff metric, and  both $B$ and $C$ contain at least two points. If there exists a  doubly connected component $A$  of $\P \sm \{B\cup C\}$ then $\mod A_i \to \mod A$; otherwise $\mod A_i \to 0$.
\end{lemma}

\begin{proof} First we note that $B_i$ and $C_i$ are connected for every $i$, therefore $B$ and $C$ are connected. This immediately implies that there exists at most one doubly connected component $A$ in the complement of $B$ and $C$. 

Suppose $A$ exists. We may choose a compactly contained in $A$ and homotopic to $A$ sub-annulus $A'$ such that $\mod A- \mod A'=\eps$ where $\eps>0$ is arbitrarily small. Then the two complementary components $B'$ and $C'$ of $A'$ are compact neighborhoods of $B$ and $C$ respectively. For $i$ large enough, we have that $B_i \subset B'$ and $C_i \subset C'$ and, hence, $A_i \supset A'$. We see that $\mod A_i > \mod A' =\mod A -\eps$. We conclude that $\liminf (\mod A_i) \ge \mod A$.


Let  $b_1, b_2 \in B$ and $c \in C$ be three distinct points in $\P$.
For $i$ large enough, there exist three distinct points $b_1^i,b_2^i$ and $c^i$ such that $b_1^i, b_2^i \in B_i$ and $c^i \in C_i$, and $d(b_1,b_1^i)<\eps$, $d(b_2,b_2^i)<\eps$ and $d(c,c^i)<\eps$. Letting $\eps$ go to 0, we construct a sequence $\{m_i\}$ of Moebius transformations such that $b_1,b_2 \in m_i(B_i)$,
$c \in m_i(C_i)$ and $\{m_i\}$ converges uniformly to the identity map on $\P$. Then the Hausdorff distance between $B_i$ and $m_i(B_i)$ tends to 0 as $i\to\infty$, and $m_i(B_i) \to B$ and, analogously,  $m_i(C_i) \to C$. Since $\mod m_i(A_i) = \mod A_i$, it is enough to prove the statement of the lemma for the sequence $\{m_i(A_i)\}$. 

Thus, we may assume that $b_1, b_2 \in B_i$ and $c\in C_i$ for all $i$. Let $a=\limsup(\mod A_i)>0$; pick a sequence $n_i$ such that $\mod A_{n_i} > a - \eps$ for all  $i\in \N$. Consider a sequence of sub-annuli $A_{n_i}'$ of $A_{n_i}$ with  $\mod A_{n_i}' = a - \eps$ for all $i\in\N$; let $\{f_i: R \to A'_{n_i}\}$ be a sequence of conformal isomorphisms from a round annulus $R$ of modulus $a-\eps$ onto $A'_{n_i}$. Since all $f_i$ do not assume values $b_1, b_2$ or $c$ on $R$, the family $\{f_i\}$ is normal by Montel's theorem and there exists a subsequence that converges locally uniformly to a holomorphic map $f$ defined on $R$.  Clearly the diameters of sets  $B_i$ tend to the diameter of $B$  and the diameters of sets $C_i$ tend to the diameter of $C$. The core curve of $R$ is mapped by $f_i$ to a smooth Jordan curve separating $B_i$ and $C_i$ and therefore the lower limit of the diameter of this curve is positive. This yields that $f$ cannot be constant and, hence,  is a conformal map onto an annulus that separates $b_1, b_2$ and $c$. Let us prove that $f(R) \cap B = \emptyset$.

On contrary, let $f(z)=b\in B$ for some point $z\in R$. To simplify the notation,  we assume that $b\neq \infty$ and that $f_i \to f$. Let $r$ be the distance between $z$ and the boundary of the annulus $R$. For any $\eps >0$, if $i$ is large enough there exists a point $b_i' \in B_i$ such that $|b_i'-b|<\eps$ and $|f_i(z)-b|<\eps$. It follows that $|f_i(z)-b'_i|<2\eps$ but $b_i'$ is not in the image of $f_i$. Koebe 1/4 theorem implies that $f'_i(z) < 8\eps/r$. Hence,  $f'(z)=\lim f_i'(z) =0$, which contradicts the fact that $f$ is conformal. Therefore $f(R)$ does not intersect $B$ nor, by the same argument, $C$. We conclude that the component $A$ of the complement of $B\cup C$ containing $f(R)$ is an annulus and $\mod A \ge \mod f(R)=\mod R=a-\eps$. We see that $\limsup (\mod A_i) \le \mod A$.
\end{proof}

\section{Positive matrices}

\label{sec:matrices}
This work uses classical results from matrix theory (see, for example, \cite{gant}).

\begin{definition} We say that a matrix $M$ is \emph{non-negative} (or \emph{positive}) and write  $M \ge 0$ (or $M > 0$) if all its entries are non-negative (or, respectively, positive). Same definition also applies for vectors.
\end{definition}

More generally, we write $A \ge B$ (or $A > B$) if  $A - B \ge 0$ (respectively, $A-B>0$).

\begin{definition} 
A square matrix $M$ is \emph{reducible} if there exists a permutation matrix $P$ such that conjugation by $P$  puts $M$ in the block form
\begin{equation} P^{-1} M P = \left( 
\begin{array}{cc}
	M_{11}  & 0 \\
	M_{21}  & M_{22}
\end{array}
 \right),
 \label{eq:reducible}
 \end{equation}
 where $M_{11}$ and $M_{22}$ are square matrices. If no such permutation exists, the matrix $M$ is \emph{irreducible}.
\end{definition}

It is obvious that any positive matrix is irreducible. Conjugating by a permutation matrix, any matrix can be written in the form

$$ P^{-1} M P= \left( 
\begin{array}{cccc}
	M_{11}  & 0 & \ldots & 0\\
	M_{21}  & M_{22} &\ldots & 0\\
	\ldots & \ldots & \ldots & \ldots\\
	M_{k1} & M_{k2} & \ldots & M_{kk} 
\end{array}
 \right),$$
where all blocks $M_{ii}$ are square and irreducible. Clearly, the spectral radius of $M$ is the maximum of spectral radii of $M_{kk}$. 

The following theorem has numerous implications in many subjects.

\begin{theorem}[Perron-Frobenius] If $M$ is an irreducible non-negative square matrix then there exist a unique largest eigenvalue (\emph{the Perron-Frobenius or leading eigenvalue}) $\lambda(M)$ which is real and positive, and a unique up to scale positive eigenvector with eigenvalue $\lambda(M)$.
\label{thm:p-f}
\end{theorem}

\begin{corollary} 
\label{cor:largestEV}
If $M$ is a non-negative square matrix then there exist a largest eigenvalue $\lambda(M)$ which is real and positive.
\end{corollary}
\begin{proof}
  The statement follows immediately from the previous theorem and preceding remark.
\end{proof}

\begin{definition} A matrix $M$ is called \emph{primitive} if there exists a power $k \in \N$ for which $M^k$ is positive.
\end{definition}

Denote by $I_n$ the $n\times n$ identity matrix. The following is a well-known result.

\begin{theorem} If $M$ is an irreducible non-negative $n\times n$ matrix then $(I_n+M)^{n-1}>0$. In particular, the matrix $I_n+M$ is primitive. 
\end{theorem}

We prove a slightly more general statement.

\begin{proposition}
\label{prop:primitive}
If $M$ is an irreducible non-negative $n \times n$ matrix and at least one diagonal entry of $M$ is positive then $M^{2n-2}>0$ and, hence, $M$ is  primitive.
\end{proposition}

\begin{proof} Without loss of generality, we can assume that all non-zero entries of $M$ are equal to 1. Construct a directed graph $G$ with $n$ vertices  using $M$ as an adjacency matrix, i.e. adding an edge from $i$-th to $j$-th vertex if and only if the corresponding entry $m_{ij}$ is equal to 1. Since $M$ is irreducible, there exists a directed path in $G$ between any two vertices. Indeed, take any vertex $a$ and denote by $A$ the set of all vertices you can reach starting at $a$. If $A$ is not the whole set of vertices then a permutation, that puts all vertices in $A$ before the rest of the vertices, conjugates $M$ to a matrix in the block form (\ref{eq:reducible}). Note that the shortest path between any two vertices is evidently no longer than $n-1$.

We write $M^k=(m_{ij}^k)$. We notice that $m_{ij}^k$ is equal to the number of paths in $G$ of length exactly $k$  that start at $i$-th vertex and end at $j$-th vertex. Therefore, it is enough to prove that between any two points there exists a path of length $2n-2$. Recall that $M$ has a non-zero diagonal entry which corresponds to a loop in $G$ at some vertex $v$. Any two vertices $a$ and $b$ can be connected by a path of length at most $2n-2$ that passes through $v$ because there exist paths of length at most $2n-2$ connecting $a$ to $v$ and $v$ to $a$. To construct a path of length $2n-2$ we simply insert the loop at $v$ an appropriate number of times into the former path.
\end{proof}

We will also use the following statement (compare to Section XIII.5 in \cite{gant}).  

\begin{theorem}
\label{thm:imprimitive}
For any irreducible matrix $M$ there exists a power $k\in\N$ such that $M^k$, conjugating by an appropriate permutation $P$ can be written in the block diagonal form
$$ P^{-1} M^k P= \left( 
\begin{array}{cccc}
	M_{11}  & 0 & \ldots & 0\\
	0  & M_{22} &\ldots & 0\\
	\ldots & \ldots & \ldots & \ldots\\
	0 & 0 & \ldots & M_{kk} 
\end{array}
 \right),$$ where all $M_{ii}$ are positive and have the same leading eigenvalue as $M^k$.
 
\end{theorem}

\bibliographystyle{hep}
\bibliography{my}

\end{document}